\documentclass[11pt]{article}
\usepackage{amsmath}
\usepackage{amssymb}
\usepackage{amsthm}
\usepackage{cases}
\usepackage{url}

\newtheorem{theorem}{Theorem}[section]

\newtheorem{corollary}[theorem]{Corollary}
\newtheorem{proposition}[theorem]{Proposition}

\newcommand\sa{\smallskipamount}
\newcommand\sPP{\\[\sa]\indent}
\newcommand\al\alpha
\newcommand\be\beta
\newcommand\de\delta
\newcommand\tha\theta
\newcommand\la\lambda
\newcommand\La\Lambda
\newcommand\ga{\gamma}
\newcommand\Ga{\Gamma}
\begin{document}
\title{The higher-order differential operator for the generalized Jacobi polynomials - new representation and symmetry}
\author{Clemens Markett}
\date{}
\maketitle

\numberwithin{equation}{section}
\numberwithin{theorem}{section}
\begin{abstract}
For a long time it has been a challenging goal to identify all orthogonal polynomial systems that occur as eigenfunctions of a linear differential equation. One of the widest classes of such eigenfunctions known so far, is given by Koornwinder's generalized Jacobi polynomials with four parameters $\al,\be\in\mathbb{N}_{0}$ and $M,N \ge 0$ determining the orthogonality measure on the interval $-1 \le x \le 1$. The corresponding differential equation of order $2\al+2\be+6$ is presented here as a linear combination of four elementary components which make the corresponding differential operator widely accessible for applications. In particular, we show that this operator is symmetric with respect to the underlying scalar product and thus verify the orthogonality of the eigenfunctions.
\end{abstract}
\section{Introduction}
\label{intro}
In 1999, J. and R. Koekoek \cite{KK2} found a new class of higher-order linear differential equations satisfied by the generalized Jacobi polynomials $\{P_{n}^{\al,\be,M,N}(x)\} _{n=0}^{\infty},\;\al,\be >-1,\;M,N \ge 0$. These function systems were introduced and investigated by T. H. Koornwinder \cite{Ko}
as the orthogonal polynomials with respect to a linear combination of the Jacobi weight function $w_{\al,\be}$ and, in general, two delta “functions” at the endpoints of the interval $-1 \le x \le 1$,
\begin{equation}
\begin{aligned}
&w_{\al,\be,M,N}(x)=w_{\al,\be}(x)+M\de(x+1)+N\de(x-1),\\ 
&w_{\al,\be}(x)=h_{\al,\be}^{-1}(1-x)^{\al}(1+x)^{\be},\\
&h_{\al,\be}=\int_{-1}^{1}(1-x)^{\al}(1+x)^{\be}dx=
2^{\al+\be+1}\Gamma(\al+1)\Ga(\be+1)/\Ga(\al+\be+2).
 \label{eq1.1}
\end{aligned}
\end{equation}

In terms of the classical Jacobi polynomials (cf. \cite [Sec. 10.8]{HTF2})
\begin{equation}
P_n^{\al,\be}(x)=\frac{(\al+1)_n}{n!}{}_2F_1\big(-n,n+\al+\be+1;\al+1;\frac{1-x}{2}\big),
n\in\mathbb{N}_{0}=\lbrace 0, 1, \cdots \rbrace,
 \label{eq1.2}
 \end{equation}
Koornwinder's generalized Jacobi polynomials are given by (see \cite{KK2},\cite{Ko},\cite{Ba})
\begin{equation}
P_n^{\al,\be,M,N}(x)=P_n^{\al,\be}(x)+MQ_n^{\al,\be}(x)
+NR_n^{\al,\be}(x)+MNS_n^{\al,\be}(x),\; n\in\mathbb{N}_{0},
 \label{eq1.3}
\end{equation}   
where $Q_0^{\al,\be}(x)=R_0^{\al,\be}(x)=
S_0^{\al,\be}(x)=S_1^{\al,\be}(x)=0$ and
\begin{equation}
 Q_n^{\al,\be}(x)=q_n^{\al,\be}(x+1)P_{n-1}^{\al,\be+2}(x),\; 
 q_n^{\al,\be}=\frac{(\al+\be+2)_n(\be+2)_{n-1}}{2n!\,(\al+1)_{n-1}}, n \ge 1,
 \label{eq1.4}
 \end{equation}
\begin{equation}
R_n^{\al,\be}(x)=r_n^{\al,\be}(x-1)P_{n-1}^{\al+2,\be}(x),\; r_n^{\al,\be}=\frac{(\al+\be+2)_n(\al+2)_{n-1}}{2n!\,(\be+1)_{n-1}}, n \ge 1,
 \label{eq1.5}
 \end{equation}	  
\begin{equation}
S_n^{\al,\be}(x)=s_n^{\al,\be}(x^2-1)P_{n-2}^{\al+2,\be+2}(x),\; s_n^{\al,\be}=\frac{(\al+\be+2)_n(\al+\be+2)_{n+1}}{(\al+1)(\be+1)4(n-1)!\,n!},
 n \ge 2.
 \label{eq1.6}
 \end{equation}	
Because of the well-known relationship $P_n^{\al,\be}(x)=(-1)^{n}P_n^{\be,\al}(-x)$, one also has
\begin{equation}
 Q_n^{\al,\be}(x)=(-1)^{n}R_n^{\be,\al}(-x),
 S_n^{\al,\be}(x)=(-1)^{n}S_n^{\be,\al}(-x)
  \label{eq1.7}
 \end{equation}	 
 and hence $P_n^{\al,\be,M,N}(x)=
 (-1)^{n}P_n^{\be,\al,N,M}(-x),\;n\in\mathbb{N}_{0}$.
 
Shortly after Koekoek's discovery, the generalized Jacobi equation was interpreted by Bavinck \cite{Ba} as an eigenvalue equation with the generalized Jacobi polynomials as its eigenfunctions. According to the four terms of definition (1.3), Bavinck looked for a representation in the modular form 
\begin{equation}
\begin{aligned}
&[L_x^{\al,\be}+MA_x^{\al,\be}
+NB_x^{\al,\be}+MNC_x^{\al,\be}]P_n^{\al,\be,M,N}(x)\\
&=[\la_n+M\al_n+N\be_n+MN\ga_n]P_n^{\al,\be,M,N}(x),\; n\in\mathbb{N}_{0}.
\end{aligned}
 \label{eq1.8}
\end{equation}
Notice that for $M=N=0$, equation (1.8) comprises the classical second-order differential equation for the Jacobi polynomials,
$L_x^{\al,\be}P_n^{\al,\be}(x)=\la_nP_n^{\al,\be}(x),\,n\in\mathbb{N}_{0}$, where 
 \begin{equation}
   L_{x}^{\al,\be}=(x^2-1)D_x^2+\left[\al-\be+(\al+\be+2)x\right]D_x,\; \la_n=n(n+\al+\be+1).
 \label{eq1.9}
 \end{equation}
 Throughout this paper, $D_x^i\equiv (D_x)^i$, $i=1,2,\cdots$, denotes an $i$-fold differentiation with respect to $x$. The other three differential expressions on the left-hand side of (1.8), $A_x^{\al,\be},B_x^{\al,\be},C_x^{\al,\be}$, were explicitly given by J. and R. Koekoek \cite{KK2} as polynomials in $D_x$ with coefficient functions involving certain generalized hypergeometric sums. In addition, Bavinck \cite{Ba} used some operator theoretical arguments to present them in “factorized” form, i. e. as a product of second-order differential operators, see (2.20) below. In particular, it turned out that
 \begin{equation}
 \begin{aligned}
 &A_x^{\al,\be}\text{ is of order }2\be+4,\text{ if }\be \in\mathbb{N}_{0},\al >-1,\\
 &B_x^{\al,\be}\text{ is of order }2\al+4,\text{ if }\al \in\mathbb{N}_{0},\be >-1,\\
 &C_x^{\al,\be}\text{ is of order }2\al+2\be+6,\text{ if }\al,\be \in\mathbb{N}_{0}.
 \end{aligned}
  \label{eq1.10}
 \end{equation}
 
 The main purpose of this paper is to establish another, completely elementary representation of the higher-order differential equation (1.8). This is carried out in Section 2 together with some of its features. The crucial part of the proof is to handle the operator $C_x^{\al,\be}$  and its interplay with the other three components. Among others, it illustrates why, for different values of the parameters $\al,\be\in\mathbb{N}_{0}, M,N > 0$, the order of the equation is
 $2\al+2\be+6$ and thus less than the sum of the two orders $2\al+4$ and $2\be+4$ in the particular Jacobi-type cases with only one additional point mass in the weight function. New elementary representations of the latter Jacobi-type equations as well as of the symmetric ultrapherical-type equation for $\al =\be \in\mathbb{N}_{0}$, $M=N > 0$, have recently been given by the author \cite{Ma2},\cite{Ma1}. Moreover, our present results are suited to gain new insight into earlier contributions of various authors in the field since the pioneering work of H. L. Krall \cite{Kr2}, see also \cite{Kr1},\cite{Koe}. A prominent example is the generalized Legendre case $\al=\be=0$ with different values of $M,N > 0$. The corresponding sixth-order equation for the so-called Krall polynomials has been discovered by L.L. Littlejohn \cite{Li}. Its new elementary form will be presented in Corollary 2.4 below.
  
 In Section 3 we use the new representation of the generalized Jacobi equation to show that the corresponding differential operator (3.3) is symmetric with respect to the weighted scalar product associated with $w_{\al,\be,M,N}(x)$. Consequently, its eigenfunctions form an orthogonal system in the corresponding function space. Moreover, our results may lay foundation to a deeper spectral theoretical treatment of the higher-order differential equation. 
 
   \section{Elementary representation of the generalized Jacobi equation}
 \label{sec:2}
 
 First of all, we slightly reformulate equation (1.8), which is motivated as follows. While the eigenvalue component $\la_n=n(n+\al+\be+1)$ is a monic polynomial in $n$ of second degree, say $\La_{2,n}^{\al,\be}$ , the other three eigenvalue components $\al _n, \be _n, \ga _n$  were given in \cite[(2.2)-(2.4)]{Ba} as certain algebraic expressions. However, each of them can be normalized to become a monic polynomial in $n$, as well, where its degree coincides with the order of the corresponding differential operator. Adopting the notations used in \cite{Ma2} and indicating the respective order as another index, we write 
  \begin{equation*}
     \al _n=b_{\be,\al}^{-1}\La_{2\be+4,n}^{\be,\al},\;
      \be _n=b_{\al ,\be}^{-1}\La_{2\al +4,n}^{\al,\be},\;
      \ga _n=c_{\al ,\be}^{-1}\La_{2\al +2\be +6,n}^{\al,\be},
      \end{equation*}
 where
  \begin{equation}
   b_{\al,\be}=(\al+2)!(\be+1)_{\al+1},\;
   c_{\al,\be}=(\al+1)(\be+1)(\al+\be+3)(\al+\be+1)!^2,  
    \label{eq2.1}
   \end{equation}
  and
   \begin{equation}
    \La_{2\al+4,n}^{\al,\be}=(n)_{\al+2}(n+\be)_{\al+2},\;
    \La_{2\al+2\be+6,n}^{\al,\be}=(n-1)_{\al+\be+3}(n)_{\al+\be+3}.
     \label{eq2.2}
    \end{equation}
  Similarly, we normalize the differential operators on the left-hand side of (1.8) by setting
  \begin{equation}
  L_{2,x}^{\al,\be} \simeq L_x^{\al,\be},\;
  \widetilde{L}_{2\be+4,x}^{\be,\al} \simeq b_{\be,\al}A_x^{\al,\be},\;  
  L_{2\al+4,x}^{\al,\be}\simeq b_{\al,\be}B_x^{\al,\be},\;\\  
  L_{2\al+2\be+6,x}^{\al,\be}\simeq c_{\al,\be}C_x^{\al,\be}.
       \label{eq2.3}
      \end{equation}
  \begin{theorem}
 \label{thm2.1}
 For $\al,\be\in\mathbb{N}_0,\,M,N>0$, 
 Koorwinder's generalized Jacobi polynomials\\ $y_n(x)=P_{n}^{\al,\be,M,N}(x), n \in \mathbb{N}_0$, are the eigenfunctions of the generalized Jacobi equation
  \begin{equation}
  \begin{aligned}
  &\bigg\lbrace \big\lbrack L_{2,x}^{\al,\be}-\La_{2,n}^{\al,\be}\big\rbrack +  
   \frac{M}{b_{\be,\al}}\big\lbrack\widetilde{L}_{2\be+4,x}^{\be,\al}
       -\La_{2\be+4,n}^{\be,\al}\big\rbrack+\\
   &\frac{N}{b_{\al,\be}}\big\lbrack L_{2\al+4,x}^{\al,\be}
         -\La_{2\al+4,n}^{\al,\be}\big\rbrack +
    \frac{MN}{c_{\al,\be}}\big\lbrack L_{2\al+2\be +6,x}^{\al,\be}
         -\La_{2\al+2\be +6,n}^{\al,\be}\big\rbrack \bigg\rbrace y_n(x)=0.
        \label{eq2.4}
    \end{aligned}
    \end{equation}
 Here, the four eigenvalue components are given in (1.9) and (2.2) and, for any sufficiently smooth functions $y(x)$,
  \begin{equation}
   L_{2,x}^{\al,\be}y(x) =\frac{1}{(x-1)^{\al}(x+1)^{\be}}D_x\big\lbrack (x-1)^{\al +1}(x+1)^{\be +1}D_x y(x)\big\rbrack,
    \label{eq2.5}
  \end{equation}
  \begin{equation}
    \widetilde{L}_{2\be +4,x}^{\be,\al}y(x) =\frac{x+1}{(x-1)^{\al}}D_x^{\be +2}\big\lbrace (x-1)^{\al +\be+2}D_x^{\be +2}\big\lbrack (x+1)^{\be +1} y(x)\big\rbrack \big\rbrace,
      \label{eq2.6}
    \end{equation}
     \begin{equation} 
    L_{2\al +4,x}^{\al ,\be}y(x) =\frac{x-1}{(x+1)^{\be}}D_x^{\al+2}\big\lbrace (x+1)^{\al +\be+2}D_x^{\al+2} \big\lbrack (x-1)^{\al+1} y(x)\big\rbrack \big\rbrace,
         \label{eq2.7}
     \end{equation}     
     \begin{equation}
      L_{2\al+2\be +6,x}^{\al,\be}y(x) =(x^2-1)D_x^{\al+\be+3}\big\lbrace (x-1)^{\be+1}(x+1)^{\al+1}D_x^{\al+\be+3} \big\lbrack (x-1)^{\al+1} (x+1)^{\be+1} y(x)\big\rbrack \big\rbrace.
         \label{eq2.8}
      \end{equation}    
 Expanding these four differential expressions as power series in $D_x^{i}$, the coefficient functions are polynomials of degree $i$ in $x$, where those of highest degree are of a particularly simple form, i.e.,
 \begin{equation}
   \begin{aligned}
   L_{2,x}^{\al,\be}y(x)
    =&\sum_{i=1}^{2}c_i^{\al,\be}(x)D_x^{i}y(x),\;
   c_1^{\al,\be}(x)=\al-\be+(\al+\be+2)x,\;c_2^{\al,\be}(x)=x^2-1\\
  \widetilde{L}_{2\be+4,x}^{\be,\al}y(x)
    =&\sum_{i=1}^{2\be+4}\widetilde{d}_i^{\be,\al}(x)D_x^{i},\;
    \widetilde{d}_{2\be+4}^{\be,\al}(x)=(x^2-1)^{\be+2}\\
  L_{2\al+4,x}^{\al,\be}y(x)
     =&\sum_{i=1}^{2\al+4}d_i^{\al,\be}(x)D_x^{i}y(x),\;
     d_{2\al+4}^{\al,\be}(x)=(x^2-1)^{\al+2}\\   
  L_{2\al+2\be+6,x}^{\al,\be}y(x)
       =&\sum_{i=1}^{2\al+2\be +6}e_i^{\al,\be}(x)D_x^{i}y(x),\;
       e_{2\al+2\be+6}^{\al,\be}(x)=(x^2-1)^{\al+\be+3}.    
           \label{eq2.9}
     \end{aligned}
     \end{equation}
   \end{theorem}
 
  The proof of Theorem 2.1 crucially depends on the following properties of the four components of equation (2.4). 
  \begin{proposition}	
 \label{prop2.2}
  Let $\al,\be\in\mathbb{N}_0$. The four terms (1.2), (1.4)\textendash (1.6) of Koornwinder's generalized Jacobi polynomials $P_{n}^{\al,\be,M,N}(x)$, $n \in \mathbb{N}_0$, are characterized as polynomial solutions of the four eigenvalue equations  
  \begin{equation}
     L_{2,x}^{\al,\be}P_n^{\al,\be}(x)=\La_{2,n}^{\al,\be}P_n^{\al,\be}(x)
      \label{eq2.10}
    \end{equation}
    \begin{equation}
      \widetilde{L}_{2\be+4,x}^{\be,\al}Q_n^{\al,\be}(x) =
      \La_{2\be+4,n}^{\be,\al}Q_n^{\al,\be}(x)
        \label{eq2.11}
      \end{equation}
       \begin{equation} 
      L_{2\al+4,x}^{\al,\be}R_n^{\al,\be}(x)=
       \La_{2\al+4,n}^{\al,\be}R_n^{\al,\be}(x)
           \label{eq2.12}
       \end{equation}     
       \begin{equation}
        L_{2\al+2\be+6,x}^{\al,\be}S_n^{\al,\be}(x) =
         \La_{2\al+2\be+6,n}^{\al,\be}S_n^{\al,\be}(x).
           \label{eq2.13}
        \end{equation}    
  \end{proposition}	
 \begin{proof} 
 A combination of the two differentiation formulas (cf. \cite [10.8.(17),(38)]{HTF2})
  \begin{equation}
  D_x P_n^{\ga,\de}(x)=
  \tfrac{1}{2}(n+\ga+\de+1)P_{n-1}^{\ga+1,\de+1}(x),\;\ga,\de >-1
       \label{eq2.14}
     \end{equation}
  \begin{equation}
  D_x [(x-1)^{\ga}(x+1)^{\de}P_n^{\ga,\de}(x)]=
   2(n+1)(x-1)^{\ga-1}(x+1)^{\de-1}P_{n+1}^{\ga-1,\de-1}(x),\;\ga,\de >0
         \label{eq2.15}
       \end{equation}
  readily yields the classical Jacobi equation (2.10) via
  \begin{equation*}
  \begin{aligned}
     L_{2,x}^{\al,\be}P_n^{\al,\be}(x) =&
          \frac{n+\al+\be+1}{2(x-1)^{\al}(x+1)^{\be}}D_x\big\lbrack (x-1)^{\al +1}(x+1)^{\be+1}P_{n-1}^{\al+1,\be+1}(x)\big\rbrack\\
          =&n(n+\al+\be+1)P_n^{\al,\be}(x)=\La_{2,n}^{\al,\be} P_n^{\al,\be}(x).
    \end{aligned}
    \end{equation*}
   In order to carry out the higher-order derivatives occurring in (2.6)\textendash (2.8), we iteratively use two further differentiation formulas which may be derived from standard properties of the Jacobi polynomials (cf.\cite [10.8]{HTF2}), 
    \begin{equation}
     D_x \big\lbrack (x-1)^{\ga}P_n^{\ga,\de}(x)\big\rbrack =
     (n+\ga)(x-1)^{\ga-1}P_n^{\ga-1,\de+1}(x),\;\ga >0,\;\de >-1
          \label{eq2.16}
        \end{equation}
     \begin{equation}
     D_x \big\lbrack (x+1)^{\de}P_n^{\ga,\de}(x)\big\rbrack =
      (n+\de)(x+1)^{\de-1}P_n^{\ga+1,\de-1}(x),\;\ga >-1,\;\de >0.
            \label{eq2.17}
          \end{equation} 
   Concerning (2.11) it follows that for each $n\in\mathbb{N}$ , 
   \begin{equation*}
    \begin{aligned}
    \widetilde{L}_{2\be+4,x}^{\be,\al}Q_n^{\al,\be}(x) =&
   \frac{x+1}{(x-1)^{\al}}D_x^{\be+2}\big\lbrace (x-1)^{\al+\be+2}D_x^{\be+2} \big\lbrack (x+1)^{\be+2} q_n^{\al,\be} P_{n-1}^{\al,\be+2}(x)\big\rbrack \big\rbrace\\
    =& q_n^{\al,\be} \frac{x+1}{(x-1)^{\al}} D_x^{\be+2}\big\lbrace (x-1)^{\al+\be+2}(n)_{\be+2} P_{n-1}^{\al+\be+2,0}(x)\big\rbrace\\
   =& q_n^{\al,\be}(x+1)(n)_{\be+2}(n+\al)_{\be+2}\, 
   P_{n-1}^{\al,\be+2}(x)= \La_{2\be+4,n}^{\be,\al}Q_n^{\al,\be}(x).             \end{aligned}
    \end{equation*}
   Similarly, (2.12) is achieved by (cf.\cite[(4.7)]{Ma2})
   \begin{equation*}
       \begin{aligned}
       L_{2\al+4,x}^{\al,\be}R_n^{\al,\be}(x) =&
      \frac{x-1}{(x+1)^{\be}}D_x^{\al+2}\big\lbrace (x+1)^{\al+\be+2}D_x^{\al+2} \big\lbrack (x-1)^{\al+2} r_n^{\al,\be} P_{n-1}^{\al+2,\be}(x)\big\rbrack \big\rbrace\\
       =& r_n^{\al,\be} \frac{x-1}{(x+1)^{\be}} D_x^{\al+2}\big\lbrace (x+1)^{\al+\be+2}(n)_{\al+2} P_{n-1}^{0,\al+\be+2}(x)\big\rbrace\\
      =& r_n^{\al,\be}(x-1)(n)_{\al+2}(n+\be)_{\al+2}\, 
      P_{n-1}^{\al+2,\be}(x)= \La_{2\al+4,n}^{\al,\be}R_n^{\al,\be}(x).             \end{aligned}
       \end{equation*}
   Concerning equation (2.13), we first apply formula (2.15), then (2.16) if $\al \ge\be$ [or (2.17) if $\al < \be$], and finally (2.14) to obtain, for $n \ge 2$,
    \begin{equation}
     \begin{aligned}
     D&_x^{\al+\be+3}\big\lbrack (x-1)^{\al+2}(x+1)^{\be+2}
     P_{n-2}^{\al+2,\be+2}(x)\big\rbrack\\
      =& 2^{\be+2}(n-1)_{\be+2}D_x^{\al+1}\big\lbrack (x-1)^{\al-\be} P_{n+\be}^{\al-\be,0}(x)\big\rbrack\\
      =& 2^{\be+2}(n-1)_{\be+2}(n+\be+1)_{\al-\be}D_x^{\be+1}
      \big\lbrack P_{n+\be}^{0,\al-\be}(x)\big\rbrack\\
      =& 2^{\be+2}(n-1)_{\al+2}2^{-\be-1}(n+\al+1)_{\be+1} P_{n-1}^{\be+1,\al+1}(x)\\ 
      =& 2(n-1)_{\al+\be+3}P_{n-1}^{\be+1,\al+1}(x).
      \end{aligned}
      \label{eq2.18}   
      \end{equation}   
   Moreover, in view of (2.18) and (2.14), it follows that
   \begin{equation}
        \begin{aligned}
    D&_x^{\al+\be+3}\big\lbrack (x-1)^{\be+1}(x+1)^{\al+1}
    P_{n-1}^{\be+1,\al+1}(x)\big\rbrack\\
    =& 2(n)_{\al+\be+1}D_x^2 P_n^{\al,\be}(x)
    = \tfrac{1}{2}(n)_{\al+\be+3} P_{n-2}^{\al+2,\be+2}(x).
    \end{aligned}
    \label{eq2.19}   
     \end{equation}   
   Putting together the two identities (2.18) and (2.19), we arrive at the required identity
   \begin{equation*}
   \begin{aligned}
   L&_{2\al+2\be +6,x}^{\al,\be}S_n^{\al,\be}(x)\\
   =&(x^2-1)D_x^{\al+\be+3}\big\lbrace (x-1)^{\be+1} (x+1)^{\al+1}D_x^{\al+\be+3} \big\lbrack (x-1)^{\al+2} (x+1)^{\be+2} s_n^{\al,\be} P_{n-2}^{\al+2,\be+2}(x)\big\rbrack\big\rbrace\\
   =&s_n^{\al,\be}(x^2-1) 2(n-1)_{\al+\be+3} D_x^{\al+\be+3}\big\lbrace (x-1)^{\be+1}(x+1)^{\al+1} P_{n-1}^{\be+1,\al+1}(x)\big\rbrace\\ 
   =&s_n^{\al,\be}(x^2-1)(n-1)_{\al+\be+3}(n)_{\al+\be+3}\,  P_{n-2}^{\al+2,\be+2}(x)= \La_{2\al+2\be+6,n}^{\al,\be}S_n^{\al,\be}(x). 
   \end{aligned}
   \end{equation*}   
     \end{proof} 
  At this stage, we recognize that certain counterparts of the identities (2.11)\textendash (2.13) are achieved via the following factorization of the three higher-order differential operators in equation (1.8). Recalling the values of the constants (2.1), there hold (see \cite[Sec. 2]{Ba})
   \begin{equation}
    \begin{aligned}
  A_x^{\al,\be}y(x)&=b_{\be,\al}^{-1}\prod_{j=0}^{\be+1}\big\lbrace L_{2,x}^{\al,\be}+\frac{2(\be+1)}{x+1}+j(\al+\be+1-j)\big\rbrace y(x),\\
  B_x^{\al,\be}y(x)&=b_{\al,\be}^{-1}\prod_{j=0}^{\al+1}\big\lbrace L_{2,x}^{\al,\be}-\frac{2(\al+1)}{x-1}+j(\al+\be+1-j)\big\rbrace y(x),\\     
  C_x^{\al,\be}y(x)&=c_{\al,\be}^{-1}\prod_{j=0}^{\al+\be +2}\big\lbrace L_{2,x}^{\al,\be}+\frac{2(\be+1)}{x+1}-\frac{2(\al+1)}{x-1}
  +j(\al+\be+1-j)\big\rbrace y(x).      
    \end{aligned}
     \label{eq2.20}   
     \end{equation}    
   \begin{proposition}	
   \label{prop2.3}
   \cite[Sec. 3]{Ba} The three operators defined in (2.20) satisfy 
         \begin{equation}
   b_{\be,\al} A_x^{\al,\be} Q_n^{\al,\be}(x) =
    \La_{2\be+4,n}^{\be,\al}Q_n^{\al,\be}(x),\;n \ge 1,
          \label{eq2.21}
        \end{equation}
         \begin{equation} 
   b_{\al,\be} B_x^{\al,\be} R_n^{\al,\be}(x) =
     \La_{2\al+4,n}^{\al,\be}R_n^{\al,\be}(x)\;n \ge 1,
          \label{eq2.22}
        \end{equation}     
       \begin{equation}
   c_{\al,\be} C_x^{\al,\be} S_n^{\al,\be}(x) =
     \La_{2\al+2\be+6,n}^{\al,\be}S_n^{\al,\be}(x)\;n \ge 2.
          \label{eq2.23}
          \end{equation}    
    \end{proposition}	 
   \begin{proof} 
   It suffices to recall the proof of (2.23). The other two identities hold analogously, see, e.g., \cite[Cor. 4.1]{Ma2}. A successive application of the $\al+\be+3$ factors in the product representing $C_x^{\al,\be}$ as well as of the Jacobi equation (2.10) with the parameters shifted to $\al+2,\be+2$ yields
   \begin{equation*}
   \begin{aligned}
   &c_{\al,\be} C_x^{\al,\be}  S_n^{\al,\be}(x)\\
   &=\prod_{j=0}^{\al+\be+2}\big\lbrace L_{2,x}^{\al,\be}+\frac{2(\be+1)}{x+1}-\frac{2(\al+1)}{x-1}
     +j(\al+\be+1-j)\big\rbrace \lbrack s_n^{\al,\be}(x^2-1)P_{n-2}^{\al+2,\be+2}(x)\rbrack\\
   &=s_n^{\al,\be}(x^2-1)\prod_{j=0}^{\al+\be+2}\big\lbrace L_{2,x}^{\al+2,\be+2}
        +(j+2)(\al+\be+3-j)\big\rbrace P_{n-2}^{\al+2,\be+2}(x)\\
   &=s_n^{\al,\be}(x^2-1)\prod_{j=0}^{\al+\be+2}\big\lbrace (n-2)(n+\al+\be+3)+
           (j+2)(\al+\be+3-j)\big\rbrace \cdot P_{n-2}^{\al+2,\be+2}(x)\\
    &=\prod_{j=0}^{\al+\be+2}\big\lbrace (n+j)(n+\al+\be+1-j)
     \big\rbrace \cdot s_n^{\al,\be}(x^2-1)P_{n-2}^{\al+2,\be+2}(x)\\
    &=(n-1)_{\al+\be+3}(n)_{\al+\be+3}\,S_n^{\al,\be}(x). 
    \end{aligned}
    \end{equation*}   
   \end{proof} 
  \noindent {\itshape Proof of Theorem 2.1.} Comparing the results of the two Propositions 2.2 and 2.3 we find that 
   \begin{equation}
   \begin{aligned}
   &\big\lbrack \widetilde{L}_{2\be+4,x}^{\be,\al}-
      b_{\be,\al} A_x^{\al,\be}\big\rbrack Q_n^{\al,\be}(x)=0,\;n \ge 1,\\
   &\big\lbrack L_{2\al+4,x}^{\al,\be}-
      b_{\al,\be} B_x^{\al,\be}\big\rbrack R_n^{\al,\be}(x)=0,\;n \ge 1,\\
  &\big\lbrack  L_{2\al+2\be +6,x}^{\al,\be}-
      c_{\al,\be} C_x^{\al,\be}\big\rbrack S_n^{\al,\be}(x)=0,\;n \ge 2.
    \end{aligned}
     \label{eq2.24}
      \end{equation}    
  In order to show that in all three cases, the difference of the two operators on the left-hand side vanishes on the space of all polynomials and thus justify the original setting (2.3), we follow the arguments used in \cite{Ba}. In fact, the polynomials $\{Q_n^{\al,\be}\}_{n=1}^{\infty}, \{R_n^{\al,\be}\}_{n=1}^{\infty}$, and $\{S_n^{\al,\be}\}_{n=2}^{\infty}$ form three bases for the polynomials divisible by $x+1, x-1$, and $x^2-1$, respectively. So it remains to verify that, similarly to $A_x^{\al,\be}, B_x^{\al,\be}, C_x^{\al,\be}$, the differential operators $\widetilde{L}_{2\be+4,x}^{\be,\al}, L_{2\al+4,x}^{\al,\be}$ map constants into $0$, while $L_{2\al+2\be +6,x}^{\al,\be}$ does it with any linear function. This, however, follows easily by definition (2.6)\textendash(2.8). In particular, if $y(x)$ is a linear function, the inner derivative in (2.8) is a constant, so that the outer derivative annihilates the whole expression. Hence, our representation (2.4) of the generalized Jacobi equation matches with the one Bavinck stated in \cite [(1.8),(2.19)]{Ba}.\\       
  The expansions (2.9) as well as their highest coefficient functions follow directly from the representations (2.5)\textendash(2.8). This concludes the proof of Theorem 2.1.\hfill $\Box$   
 \begin{corollary}
  \label{cor2.4}
  The sixth-order differential equation satisfied by the Krall polynomials $y_n(x)=P_n^{0,0,M,N}(x)$, $M,N>0$, $n\in\mathbb{N}_0$, takes the elementary form
  \begin{equation}
  \begin{aligned}
  D_x\big\lbrack(x^2-1)D_x y_n(x)\big\rbrack &+\frac{M}{2}(x+1)D_x^2 \big \lbrace (x-1)^2 D_x^2 \big\lbrack (x+1)y_n(x)\big\rbrack \big\rbrace \\
  &+\frac{N}{2}(x-1)D_x^2 \big\lbrace (x+1)^2 D_x^2\big\lbrack (x-1)y_n(x)\big\rbrack\big \rbrace\\
  &+\frac{MN}{3}(x^2-1)D_x^3 \big\lbrace (x^2-1) D_x^3 \big\lbrack (x^2-1)y_n(x)\big\rbrack\big\rbrace\\  
  &=(n)_2\big\lbrace 1+\frac{M+N}{2}(n)_2+\frac{MN}{3}(n-1)_4\big\rbrace
   y_n(x).
   \end{aligned}
   \label{eq2.25}
   \end{equation}    
   \end{corollary}
       
    Inserting the decomposition (1.3) of the generalized Jacobi polynomials, 
    \begin{equation*}
    y_n(x)= P_n^{\al,\be,M,N}(x)=P_n^{\al,\be}(x)+M\,Q_n^{\al,\be}(x)
    +N\,R_n^{\al,\be}(x)+MN\,S_n^{\al,\be}(x), \; n\in\mathbb{N}_{0}, 
     \end{equation*}	
   into the equation (2.4), we easily see that all terms of equal factors of the parameters $M$ and $N$ must vanish simultaneously. So in addition to the four identities stated in Proposition 2.2, we obtain four two-term identities as well as one with four terms. 
    \begin{corollary}
       \label{cor2.5}
     For $\al,\be\in\mathbb{N}_0,\;M,N>0$, and all $n\in\mathbb{N}_0$,
      \begin{equation}
     \big\lbrack L_{2,x}^{\al,\be}-\La_{2,n}^{\al,\be}\big\rbrack Q_n^{\al,\be}(x)+
     b_{\be ,\al}^{-1}\big\lbrack \widetilde{L}_{2\be+4,x}^{\be,\al}- \La_{2\be+4,n}^{\be,\al}\big\rbrack P_n^{\al,\be}(x)=0,   
         \label{eq2.26}
         \end{equation}
     \begin{equation}
      \big\lbrack L_{2,x}^{\al,\be}-\La_{2,n}^{\al,\be}\big\rbrack R_n^{\al,\be}(x)+
      b_{\al,\be}^{-1}\big\lbrack L_{2\al+4,x}^{\al,\be}- \La_{2\al+4,n}^{\al,\be}\big\rbrack P_n^{\al,\be}(x)=0,   
       \label{eq2.27}
       \end{equation}      
     \begin{equation}
     b_{\be,\al}^{-1}\big\lbrack \widetilde{L}_{2\be+4,x}^{\be,\al}- \La_{2\be+4,n}^{\be,\al}\big\rbrack S_n^{\al,\be}(x)+
     c_{\al,\be}^{-1}\big\lbrack L_{2\al+2\be+6,x}^{\al,\be}-
     \La_{2\al+2\be+6,n}^{\al,\be}\big\rbrack Q_n^{\al,\be}(x)=0,
     \label{eq2.28}
      \end{equation}      
     \begin{equation}
     b_{\al,\be}^{-1}\big\lbrack L_{2\al+4,x}^{\al,\be}- \La_{2\al+4,n}^{\al,\be}\big\rbrack S_n^{\al,\be}(x)+
     c_{\al,\be}^{-1}\big\lbrack L_{2\al+2\be+6,x}^{\al,\be}-
     \La_{2\al+2\be+6,n}^{\al,\be}\big\rbrack R_n^{\al,\be}(x)=0, 
      \label{eq2.29}
      \end{equation}   
     \begin{equation}
     \begin{aligned}     
    &\big\lbrack L_{2,x}^{\al,\be}-\La_{2,n}^{\al,\be}\big\rbrack S_n^{\al,\be}(x)+
    b_{\be,\al}^{-1}\big\lbrack \widetilde{L}_{2\be+4,x}^{\be,\al}- \La_{2\be+4,n}^{\be,\al}\big\rbrack R_n^{\al,\be}(x)+\\
    &b_{\al,\be}^{-1}\big\lbrack L_{2\al+4,x}^{\al,\be}- \La_{2\al+4,n}^{\al,\be}\big\rbrack Q_n^{\al,\be}(x)+
    c_{\al,\be}^{-1}\big\lbrack L_{2\al+2\be+6,x}^{\al,\be}-
     \La_{2\al+2\be+6,n}^{\al,\be}\big\rbrack P_n^{\al,\be}(x)=0. 
    \label{eq2.30} 
     \end{aligned}
     \end{equation}   
     \end{corollary}
 \noindent A direct verification of (2.27) and, likewise, of (2.26) is given in \cite[Sec. 5]{Ma2}. Moreover, we verified the five identities (2.26)\textendash (2.30) for various values of the parameters via MAPLE.\\
 
 Recently, A. J. Dur\'{a}n \cite{Du} developed a concept based on so-called $\mathcal{D}$-operators to construct orthogonal polynomial systems that are eigenfunctions of higher-order difference and differential equations. A particular example deals with the Jacobi-type equation with parameters $\al,\be$ and $M>0$. Adjusting the notations used in \cite [Sec. 8.2]{Du} to ours, these polynomials are determined as certain linear combinations of the Jacobi polynomials $P_{n-1}^{\al,\be+1}(x)$  and $P_n^{\al,\be+1}(x)$, which coincide with the Jacobi-type polynomials (1.3) for  $N=0$. Moreover, it turns out that the resulting differential equation can be written in the form (cf.(2.4))
   
  \begin{equation}
    \bigg\lbrace \big\lbrack L_{2,x}^{\al,\be}-n(n+\al +\be +1)\big\rbrack +  
    \frac{M}{(\be+2)!(\al +1)_{\be+1}}\big\lbrack
    \widetilde{L}_{2\be+4,x}^{\be,\al}-(n)_{\be+2}(n+\al)_{\be+2}
    \big\rbrack\bigg\rbrace y_n(x)=0.     
             \label{eq2.31}
   \end{equation}
  But in contrast to the representation (2.20), the higher-order differential operator $\widetilde{L}_{2\be+4,x}^{\be,\al}$ is given now as the product 
    \begin{equation}
    \widetilde{L}_{2\be+4,x}^{\be,\al}=L_{2,x}^{\al,-1}\prod_{j=0}^{\be}
    \big\lbrace L_{2,x}^{\al,\be+1}+(\al +1+j)(\be+1-j)\big\rbrace y(x),
     \label{eq2.32}
     \end{equation}
  involving the second-order operators (cf. (1.9))
      \begin{equation*}
        \begin{aligned}     
       L_{2,x}^{\al,-1}=&(x^2-1)D_x^2+(\al+1)(1+x)D_x,\\
         L_{2,x}^{\al,\be+1}=&(x^2-1)D_x^2+\left[\al-\be -1+(\al+\be+3)x\right]D_x. 
           \end{aligned}     
         \end{equation*}
  In order to verify that this operator satisfies the eigenvalue equation (2.11), as well, we proceed from the alternative setting (cf.\cite [10.8 (33)]{HTF2})
   \begin{equation}
    Q_n^{\al,\be}(x)=q_n^{\al,\be}(x+1)P_{n-1}^{\al,\be+2}(x)=q_n^{\al,\be}\;
    \frac{2(n+\be +1)P_{n-1}^{\al,\be+1}(x)+2n P_n^{\al,\be+1}(x)}{2n+\al +\be +1}.
      \label{eq2.33}
      \end{equation}
 Now we apply each operator under the product sign in (2.32) successively to the two terms on the right-hand side of (2.33) by employing the Jacobi equation with parameters $\al$ and $\be+1$. Eventually we observe that 
 \begin{equation*}
     \begin{aligned}     
   2\,L_{2,x}^{\al,-1}P_n^{\al,\be}(x)=&(x+1)\big\lbrace (x-1)D_x
   +(\al+1)\big\rbrace (n+\al+\be+1)P_{n-1}^{\al+1,\be+1}(x)\\
   =&(n+\al+\be+1)(n+\al)(x+1)P_{n-1}^{\al,\be+2}(x).
    \end{aligned}     
      \end{equation*}
 By symmetry, the operator $L_{2\al+4,x}^{\al,\be}$ in equation (2.4) has a representation analogous to (2.32). It would be interesting to find a similar expression for the operator $L_{2\al+2\be+6,x}^{\al,\be}$, as well.
   
 \section{The orthogonality of the eigensolutions}
 \label{sec:3}

 The aim of this section is to show that for different values of the combined eigenvalue parameter 
 
  \begin{equation}
   \La_{2\al+\be+6,n}^{\al,\be,M,N}=\La_{2,n}^{\al,\be}+  
   \frac{M}{b_{\be,\al}} \La_{2\be+4,n}^{\be,\al}+
   \frac{N}{b_{\al,\be}}\La_{2\al+4,n}^{\al,\be} +
   \frac{MN}{c_{\al,\be}}\La_{2\al+2\be+6,n}^{\al,\be},\;n\in\mathbb{N}_0,
    \label{eq3.1}
   \end{equation}
 the eigenfunctions of the generalized Jacobi equation (2.4) are orthogonal with respect to the scalar product
  \begin{equation}
  \begin{aligned}     
  &(f,g)_{w(\al,\be,M,N)}=(f,g)_{w(\al,\be)}+Mf(-1)g(-1)+Nf(1)g(1),\\ 
  &(f,g)_{w(\al,\be)}=h_{\al,\be}^{-1}\int_{-1}^{1}f(x)g(x)
  (1-x)^{\al}(1+x)^{\be}dx,\;f,g \in C[-1,1].
   \label{eq3.2} 
   \end{aligned}
   \end{equation}
  This in turn, is a direct consequence of the following fundamental symmetry relation.
  
  \begin{theorem}
   \label{thm3.1}
 Let $\al,\be\in\mathbb{N}_0,\;M,N>0$, and let, for sufficiently smooth function 
  $y(x)$,
  \begin{equation}
  L_{2\al+2\be +6,x}^{\al,\be,M,N}y(x)=\big\lbrack L_{2,x}^{\al,\be}+  
   \frac{M}{b_{\be,\al}} \widetilde{L}_{2\be+4,x}^{\be,\al}+
   \frac{N}{b_{\al,\be}}L_{2\al+4,x}^{\al,\be} +
   \frac{MN}{c_{\al,\be}}L_{2\al+2\be+6,x}^{\al,\be}\big\rbrack y(x)
  \label{eq3.3}
  \end{equation}
  denote the combined differential operator in equation (2.4). Then there holds 
  \begin{equation}
   \big(L_{2\al+2\be+6,x}^{\al,\be,M,N}f,g\big)_{w(\al,\be,M,N)}=
   \big(f,L_{2\al+2\be+6,x}^{\al,\be,M,N}g\big)_{w(\al,\be,M,N)},\;
    f,g \in C^{(2\al +2\be+6)}[-1,1].
   \label{eq3.4}
    \end{equation}   
   \end{theorem}
 \begin{proposition}
  \label{prop3.2}
 For any $f,g \in C^{(2\al+2\be+6)}[-1,1]$ we define the following four integral expressions, each being symmetric in its two arguments, 
   \begin{equation*}
  \begin{aligned}     
   U^{\al,\be}(f,g)=h_{\al,\be}^{-1}&\int_{-1}^{1}f'(x)g'(x)
   (1-x)^{\al+1}(1+x)^{\be+1}dx,\\
  \widetilde{V}^{\be,\al}(f,g)=h_{\al,\be}^{-1}&\int_{-1}^{1}
  D_x^{\be+2}\big\lbrack (x+1)^{\be+1}f(x)\big\rbrack
  D_x^{\be+2}\big\lbrack (x+1)^{\be+1}g(x)\big\rbrack (1-x)^{\al+\be+2}dx,\\
  V^{\al,\be}(f,g)=h_{\al,\be}^{-1}&\int_{-1}^{1}
    D_x^{\al+2}\big\lbrack (x-1)^{\al+1}f(x)\big\rbrack
    D_x^{\al+2}\big\lbrack (x-1)^{\al+1}g(x)\big\rbrack (1+x)^{\al+\be+2}dx,\\ 
 W^{\al,\be}(f,g)=h_{\al,\be}^{-1}&\int_{-1}^{1}
     D_x^{\al+\be+3}\big\lbrack (x-1)^{\al+1}(x+1)^{\be+1}f(x)\big\rbrack \cdot\\
     &\cdot D_x^{\al+\be+3}\big\lbrack (x-1)^{\al+1}(x+1)^{\be+1}g(x)\big\rbrack
     (1-x)^{\be+1}(1+x)^{\al+1}dx.   
   \end{aligned}	 
  \end{equation*}  
 With $b_{\al,\be}$, $c_{\al,\be}$ as in (2.1), the four components of the differential operator (3.3) satisfy 
  \begin{align*}
 (i) \quad &\big(L_{2,x}^{\al,\be}f,g\big)_{w(\al,\be)}=U^{\al,\be}(f,g),\\
  (ii) \quad &\big(\widetilde{L}_{2\be+4,x}^{\be,\al}f,g\big)_{w(\al,\be)}=
  \widetilde{V}^{\be,\al}(f,g)+2(\be+1)b_{\be,\al}f'(-1)g(-1),\\
(iii) \quad &\big(L_{2\al+4,x}^{\al,\be}f,g\big)_{w(\al,\be)}=
  V^{\al,\be}(f,g)-2(\al+1)b_{\al,\be}f'(1)g(1),\\
 (iv) \quad &\big(L_{2\al+2\be+6,x}^{\al,\be}f,g\big)_{w(\al,\be)}=
   W^{\al,\be}(f,g)\\
    & \hspace{3.8cm} +\frac{c_{\al,\be}}{b_{\al,\be}}2(\be+1)_{\al+2}\,
   D_x^{\al +2}\big\lbrack (x-1)^{\al+1}f(x)\big\rbrack
   \big\vert_{x=-1}\, g(-1)\\
   & \hspace{3.8cm} -\frac{c_{\al,\be}}{b_{\be,\al}}2(\al+1)_{\be+2}\,
      D_x^{\be+2}\big\lbrack (x+1)^{\be+1}f(x)\big\rbrack\big\vert_{x=1}\, g(1),\\
 (v)\quad &L_{2,x}^{\al,\be}f(x)\big\vert_{x=-1}=-2(\be+1)f'(-1),\;
   L_{2,x}^{\al,\be}f(x)\big\vert_{x=1}=2(\al +1)f'(1),\\
   &\widetilde{L}_{2\be+4,x}^{\be,\al}f(x)\big\vert_{x=-1}=0,\;
   \widetilde{L}_{2\be+4,x}^{\be,\al}f(x)\big\vert_{x=1}=2(\al+1)_{\be+2}\,
   D_x^{\be +2}\big\lbrack (x+1)^{\be+1}f(x)\big\rbrack\big\vert_{x=1},\\
   &L_{2\al+4,x}^{\al,\be}f(x)\big\vert_{x=1}=0,\;
   L_{2\al+4,x}^{\al,\be}f(x)\big\vert_{x=-1}=-2(\be+1)_{\al+2}\,D_x^{\al+2}
   \big\lbrack(x-1)^{\al+1}f(x)\big\rbrack\big\vert_{x=-1},\\
   &L_{2\al+2\be+6,x}^{\al,\be}f(x)\big\vert_{x=\pm 1}=0.
 \end{align*}   
 \end{proposition}    
\begin{proof}
(i) Using integration by parts, the integrated term vanishes at $x=\pm 1$, so that 
 \begin{equation*}
  \begin{aligned}     
   \big(L_{2,x}^{\al,\be}f,g\big)_{w(\al,\be)}=&(-1)^{\al}
   h_{\al,\be}^{-1}\int_{-1}^{1}D_x \big\lbrack (x-1)^{\al+1}(x+1)^{\be+1}
   D_x f(x)\big\rbrack g(x)dx\\
   =& -h_{\al,\be}^{-1}(1-x)^{\al+1}(1+x)^{\be+1}f'(x)g'(x)\big\vert_{x=-1}^{x=1}+
   U^{\al,\be}(f,g)\\
   =& U^{\al,\be}(f,g).
   \end{aligned}	 
 	 \end{equation*}
(ii) In view of the representation (2.6) of $\widetilde{L}_{2\be+4,x}^{\be,\al}$, a  $(\be+2)$-fold integration by parts yields
\begin{equation*}
  \begin{aligned}     
   &h_{\al,\be}\big(\widetilde{L}_{2\be+4,x}^{\be,\al}f,g\big)_{w(\al,\be)}\\
   &=(-1)^{\al}\int_{-1}^{1}D_x^{\be+2}\big\lbrace (x-1)^{\al+\be+2}D_x^{\be+2}
   \big\lbrack (x+1)^{\be+1} f(x)\big\rbrack\big\rbrace (x+1)^{\be+1}g(x)dx\\
   &=\sum_{j=0}^{\be+1}(-1)^{\al+j}D_x^{\be+1-j}\big\lbrace (x-1)^{\al+\be+2}
   D_x^{\be+2}\big\lbrack (x+1)^{\be+1} f(x)\big\rbrack\big\rbrace D_x^{j} 
   \big\lbrack (x+1)^{\be+1}g(x)\big\rbrack \big\vert_{x=-1}^{x=1}\\
   &\quad +\int_{-1}^{1}(1-x)^{\al+\be+2} D_x^{\be+2}\big\lbrack (x+1)^{\be+1} f(x)\big\rbrack D_x^{\be+2}\big\lbrack (x+1)^{\be+1} g(x)\big\rbrack dx.    
   \end{aligned}	 
 	 \end{equation*}
Since all terms of the sum vanish up to the one for $j=\be+1$, evaluated at $x=-1$, we have 
\begin{equation*}
  \begin{aligned}     
 &\big(\widetilde{L}_{2\be+4,x}^{\be,\al}f,g\big)_{w(\al,\be)}-
 \widetilde{V}^{\be,\al}(f,g)\\
 &=h_{\al,\be}^{-1}(1-x)^{\al+\be+2} D_x^{\be+2}\big\lbrack (x+1)^{\be+1} f(x)\big\rbrack D_x^{\be+1}\big\lbrack (x+1)^{\be+1} g(x)\big\rbrack \big\vert_{x=-1}\\
 &=h_{\al,\be}^{-1}\,2^{\al+\be+2}(\be +2)!\,f'(-1)(\be+1)!\,g(-1)\\
 &=2(\be+1)b_{\be,\al}f'(-1)g(-1).   
  \end{aligned}	 
    \end{equation*}
(iii) This identity is verified analogously to item (ii), see \cite[Prop. 3.2 (i)]{Ma2}.\\
(iv) Invoking the representation (2.8) and using the abbreviation 
$v_{\al,\be}(x)=(x-1)^{\al+1}(x+1)^{\be+1}$, we obtain, now by an $(\al+\be+3)$-fold integration by parts, 
\begin{equation*}
  \begin{aligned}     
   &h_{\al,\be}\big(L_{2\al +2\be+6,x}^{\al,\be}f,g\big)_{w(\al,\be)}\\
   &=(-1)^{\al}\int_{-1}^{1} D_x^{\al+\be+3}\big\lbrace v_{\be,\al}(x)D_x^{\al+\be+3}\big\lbrack v_{\al,\be}(x)f(x)\big\rbrack
   \big\rbrace v_{\al,\be}(x)g(x)dx\\
   &=\sum_{j=0}^{\al+\be+2}(-1)^{\al +j}D_x^{\al+\be+2-j}\big\lbrace v_{\be,\al}(x)
   D_x^{\al+\be+3}\big\lbrack v_{\al,\be}(x)f(x)\big\rbrack\big\rbrace D_x^{j} 
   \big\lbrack (v_{\al,\be}(x)g(x)\big\rbrack \big\vert_{x=-1}^{x=1}\\
   &\quad +(-1)^{\be+1}\int_{-1}^{1}v_{\be,\al}(x) D_x^{\al+\be+3}\big\lbrack v_{\al,\be}(x)f(x)\big\rbrack D_x^{\al+\be+3}\big\lbrack v_{\al,\be}(x)g(x)\big\rbrack dx.    
   \end{aligned}	 
 	 \end{equation*}
Here, the only non-vanishing contributions of the sum are the term for $j=\be +1$ at $x=-1$ and the term for $j=\al+1$ at $x=1$. Hence,
\begin{equation*}
  \begin{aligned}     
   &\big(L_{2\al+2\be+6,x}^{\al,\be}f,g\big)_{w(\al,\be)}-W^{\al ,\be}(f,g)\\
   &=h_{\al,\be}^{-1}(-1)^{\al+\be}D_x^{\al +1}\big\lbrace v_{\be,\al}(x)
   D_x^{\al+\be+3}\big\lbrack v_{\al,\be}(x)f(x)\big\rbrack\big\rbrace 
   D_x^{\be+1}\big\lbrack v_{\al,\be}(x)g(x)\big\rbrack \big\vert_{x=-1}\\
   &\quad -h_{\al,\be}^{-1}D_x^{\be+1}\big\lbrace v_{\be,\al}(x)
   D_x^{\al+\be+3}\big\lbrack v_{\al,\be}(x)f(x)\big\rbrack\big\rbrace 
   D_x^{\al+1}\big\lbrack v_{\al,\be}(x)g(x)\big\rbrack\big\vert_{x=1}\\
   &=h_{\al,\be}^{-1}(\al+1)!\,2^{\be+1}\frac{(\al+\be+3)!}{(\al+2)!}
   D_x^{\al+2}\big\lbrack (x-1)^{\al+1}f(x)\big\rbrack \big\vert_{x=-1}
   2^{\al+1}(\be+1)!\,g(-1)\\
   &\quad -h_{\al,\be}^{-1}(\be+1)!\,2^{\al+1}\frac{(\al+\be+3)!}{(\be+2)!}
   D_x^{\be+2}\big\lbrack (x+1)^{\be+1}f(x)\big\rbrack \big\vert_{x=1}
   (\al+1)!\,2^{\be+1}g(1).
       \end{aligned}	 
 	 \end{equation*}
This yields the assertion (iv) since 
\begin{equation*}
  \begin{aligned}     
  h_{\al,\be}^{-1}\,2^{\al+\be+2}\,(\al+1)!\,(\be +1)!\,(\al+\be+3)!
  =& c_{\al,\be}\,b_{\al,\be}^{-1}\,2(\be+1)_{\al+2}\,(\al+2)!\\
  =& c_{\al,\be}\,b_{\be,\al}^{-1}\,2(\al+1)_{\be+2}\,(\be+2)!
  \end{aligned}	 
 \end{equation*}
(v) The values of the four differential expressions at $x=\pm 1$ follow directly by definition (2.5)\textendash (2.8). 	 
\end{proof}
\noindent {\itshape Proof of Theorem 3.1.} The left-hand side of (3.4) takes the form
 \begin{equation*}
   \begin{aligned}     
 &\big(L_{2,x}^{\al,\be}f+\frac{M}{b_{\be,\al}}\widetilde{L}_{2\be+4,x}^{\be,\al}f+\frac{N}{b_{\al,\be}}L_{2\al+4,x}^{\al,\be}f +
 \frac{MN}{c_{\al,\be}}L_{2\al+2\be+6,x}^{\al,\be}f,g\big)_{w(\al,\be)}\\
 &+M\,\big\lbrack L_{2\al+2\be+6,x}^{\al,\be,M,N}f(x)\big\rbrack \big\vert_{x=-1}\,g(-1)+N\,\big\lbrack L_{2\al+2\be+6,x}^{\al,\be,M,N}f(x)\big\rbrack \big\vert_{x=1}\,g(1).  
   \end{aligned}	 
  \end{equation*}
In view of Proposition 3.2 this gives 
 \begin{equation*}
   \begin{aligned}     
 &U^{\al,\be}(f,g)+\frac{M}{b_{\be,\al}}\widetilde{V}^{\al,\be}(f,g)+
    \frac{N}{b_{\al,\be}}V^{\al,\be}(f,g)+\frac{MN}{c_{\al,\be}}W^{\al,\be}(f,g)\\
 &+\frac{M}{b_{\be,\al}}\,2(\be+1)\,b_{\be,\al}\,f'(-1)g(-1)-
   \frac{N}{b_{\al,\be}}\,2(\al+1)\,b_{\al,\be}\,f'(1)g(1)+
      \end{aligned}	 
    \end{equation*}
 \begin{equation*}
    \begin{aligned}    
  &+\frac{MN}{c_{\al,\be}}\bigg\lbrace \frac{c_{\al,\be}}{b_{\al,\be}}\,2(\be+1)_{\al+2} D_x^{\al+2}\big\lbrack
  (x-1)^{\al+1}f(x)\big\rbrack \big\vert_{x=-1}\;g(-1)\\
  &\hspace{1,5cm}-\frac{c_{\al,\be}}{b_{\be,\al}}\,2(\al+1)_{\be+2}D_x^{\be+2}
  \big\lbrack(x+1)^{\be+1}f(x)\big\rbrack \big\vert_{x=1}\;g(1)\bigg\rbrace\\
  &+M\big\lbrace -2(\be+1)f'(-1)-\frac{N}{b_{\al,\be}}\,2(\be+1)_{\al+2}
  D_x^{\al+2}\big\lbrack (x-1)^{\al+1}f(x)\big\rbrack \big\vert_{x=-1}\big\rbrace g(-1)\\
  &+N\big\lbrace 2(\al+1)f'(1)+\frac{M}{b_{\be,\al}}\,2(\al+1)_{\be+2}
  D_x^{\be+2}\big\lbrack(x+1)^{\be+1}f(x)\big\rbrack\big\vert_{x=1}\big\rbrace g(1). 
  \end{aligned}	 
  \end{equation*}
Since all integrated terms compensate, we are left with the sum of the four expressions
\begin{equation*}
 U^{\al,\be}(f,g)+\frac{M}{b_{\be,\al}}\widetilde{V}^{\al,\be}(f,g)+
 \frac{N}{b_{\al,\be}}V^{\al,\be}(f,g)+\frac{MN}{c_{\al,\be}}W^{\al,\be}(f,g).  \end{equation*}
Here we can interchange the roles of the functions $f$ and $g$ without altering the value of the four terms. So we arrive at the required identity (3.4). \hfill
$\Box$

\begin{corollary}
\label{cor3.3}
For $\al,\be\in\mathbb{N}_0,\;M,N>0$, the polynomial eigenfunctions of equation (2.4), $y_n(x)=P_n^{\al,\be,M,N}(x)$, $n \in \mathbb{N}_0$, satisfy the orthogonality relation, for $n \ne m$,
\begin{equation}
 \begin{aligned}     
 (y_n,y_m)_{w(\al,\be,M,N)}=&h_{\al,\be}^{-1}\int_{-1}^{1} y_n(x)y_m(x)(1-x)^{\al}
 (1+x)^{\be}dx\\
 &+M\,y_n(-1)y_m(-1)+N\,y_n(1)y_m(1)=0.
 \end{aligned}  
     \label{eq3.5} 
  \end{equation}
\end{corollary}
\begin{proof}
Employing the symmetry property (3.4) of the differential operator
$L_{2\al+2\be+6,x}^{\al,\be,M,N}$, we obtain
\begin{equation}
  \begin{aligned}
 &\big(\La_{2\al+2\be+6,n}^{\al,\be,M,N}-\La_{2\al+2\be+6,m}^{\al,\be,M,N}\big)     
 \,(y_n,y_m)_{w(\al,\be,M,N)}\\
 &=\big(L_{2\al+2\be+6,x}^{\al,\be,M,N}y_n,y_m\big)_{w(\al,\be,M,N)}-
 \big(y_n,L_{2\al+2\be+6,x}^{\al,\be,M,N}y_m\big)_{w(\al,\be,M,N)}=0.
 \end{aligned}  
 \label{eq3.6} 
 \end{equation}
Since the eigenvalues $\La_{2\al+2\be+6,n}^{\al,\be,M,N}$ are strictly increasing for $n \in \mathbb{N}_0$, the scalar product on the left-hand side vanishes. 
\end{proof}

\vskip0.5cm
\begin{footnotesize}
\noindent
C. Markett, Lehrstuhl A f\"ur Mathematik, RWTH Aachen, 52056 Aachen, Germany;
\sPP
E-mail: {\tt markett@matha.rwth-aachen.de}
\end{footnotesize}

\begin{thebibliography}{10}

\bibitem{Ba} H. Bavinck, 
	{\em  A note on the Koekoek's differential equation for generalized Jacobi polynomials}, J. Comput. Appl. Math. 115 (2000), 87-92.
	 
\bibitem{Du} A. J. Dur\'{a}n, 
	{\em  Using $\mathcal{D}$-operators to construct orthogonal polynomials satisfying higher order difference or differential equations}, J. Approx. Theory 174 (2013), 10–53. 
	
\bibitem{HTF2} A. Erd\'{e}lyi, W. Magnus, F. Oberhettinger, and F.G. Tricomi,
  	{\em Higher Transcendental Functions, vol. II}, McGraw-Hill, 1953.
	
\bibitem{Koe} R. Koekoek,  
	{\em Differential equations for symmetric generalized ultraspherical polynomials}, Trans. Amer. Math. Soc. 345 (1994), 47-72. 

\bibitem{KK2} J. Koekoek and R. Koekoek, 
	{\em Differential equations for generalized Jacobi polynomials}, J. Comput. Appl. Math. 126 (2000), 1-31.
		
\bibitem{Ko} T.H. Koornwinder, 
	{\em Orthogonal polynomials with weight function $(1-x)^{\alpha}(1+x)^{\beta}+M\delta(x+1)+N\delta(x-1)$}, Canad. Math. Bull. 27(2) (1984), 205-214. 
	
\bibitem{Kr1} A.M. Krall, 
	{\em Orthogonal polynomials satisfying fourth order differential equations}, Proc. Roy. Soc. Edinburgh Sect. A 87 (1981), 271-288.

\bibitem{Kr2} H.L. Krall,
	{\em On orthogonal polynomials satisfying a certain fourth order differential equation}, Pennsylvanian State College Stud. 6, 1940.

\bibitem{Li} L.L. Littlejohn, 
	{\em The Krall polynomials: A new class of orthogonal polynomials}, Quaestiones Math. 5 (1982), 255-265.
 	
\bibitem{Ma1} C. Markett, 
	{\em New representation and factorization of the higher-order ultraspherical-type differential equations}, J. Math. Anal. Appl. 421 (2015), 244-259. 

\bibitem{Ma2} C. Markett, 
	{\em An elementary representation of the higher-order Jacobi-type differential equation}, to appear. 
	
\end{thebibliography}
\end{document}